\documentclass[11pt]{article}
\usepackage{amssymb,amsmath,latexsym,verbatim,amsthm,color,enumerate}
\allowdisplaybreaks

\begin{document}

\begin{doublespace}

\newtheorem{thm}{Theorem}[section]
\newtheorem{lemma}[thm]{Lemma}
\newtheorem{cond}[thm]{Condition}
\newtheorem{defn}[thm]{Definition}
\newtheorem{prop}[thm]{Proposition}
\newtheorem{corollary}[thm]{Corollary}
\newtheorem{remark}[thm]{Remark}
\newtheorem{example}[thm]{Example}
\newtheorem{conj}[thm]{Conjecture}
\numberwithin{equation}{section}
\def\ee{\varepsilon}
\def\qed{{\hfill $\Box$ \bigskip}}
\def\NN{{\cal N}}
\def\AA{{\cal A}}
\def\MM{{\cal M}}
\def\BB{{\cal B}}
\def\CC{{\cal C}}
\def\LL{{\cal L}}
\def\DD{{\cal D}}
\def\FF{{\cal F}}
\def\EE{{\cal E}}
\def\QQ{{\cal Q}}
\def\RR{{\mathbb R}}
\def\R{{\mathbb R}}
\def\L{{\bf L}}
\def\K{{\bf K}}
\def\S{{\bf S}}
\def\A{{\bf A}}
\def\E{{\mathbb E}}
\def\F{{\bf F}}
\def\P{{\mathbb P}}
\def\N{{\mathbb N}}
\def\eps{\varepsilon}
\def\wh{\widehat}
\def\wt{\widetilde}
\def\pf{\noindent{\bf Proof.} }
\def\beq{\begin{equation}}
\def\eeq{\end{equation}}
\def\lam{\lambda} \def\la{\lambda}
\def\H{\mathcal{H}}
\def\nn{\nonumber}
\def\L{\mathcal{L}}
\def\WU{\mathrm{WUSC}(\overline{\alpha}, \overline{\theta}, \overline{C})}
\def\WL{\mathrm{WLSC}(\underline{\alpha},\underline{\theta},\underline{c})}

\def\us{{\rm{WUSC}}}
\def\ls{{\rm{WLSC}}}
\def\dimb{\overline{\dim}_{_{\rm B}}}
\def\dimh{\dim_{_{\rm H}}}
\def\dimp{\dim_{_{\rm P}}}
\def\ov{\overline}
\def\un{\underline}
\def\al{\alpha}
\def\de{\delta}

\newcommand{\Per}{\mathrm{Per}}
\newcommand{\norm}[1]{{\lVert #1 \rVert}}

\title{\Large \bf  Uniform dimension results for the inverse 
images of symmetric L\'evy processes}
\author{Hyunchul Park, Yimin Xiao, and Xiaochuan Yang}

\date{ \today}
\maketitle

\begin{abstract}
We prove uniform Hausdorff and packing dimension results for the inverse
images of a large class of real-valued symmetric L\'evy processes. Our
main result for the Hausdorff dimension extends that of Kaufman (1985)
for Brownian motion and that of Song, Xiao, and Yang (2018) for $\alpha$-stable
L\'evy processes with $1<\alpha<2$. Along the way we also prove an upper
bound for the uniform modulus of continuity of the local times of these
processes.
\end{abstract}

\section{Introduction}\label{introduction}
The inverse images of L\'evy processes have been studied intensively. 
When $F=\{y\}$ is a single set, the question whether $\P(X^{-1}(F)\neq\emptyset)=0$ 
boils down to whether $y$ is polar for $X$, which is an important question in potential theory.
As the sample paths of $X$ are typically irregular, so the inverse images,
in particular the level sets, are of a fractal nature. It is thus of interest to determine
the fractal dimension of the set $X^{-1}(F)$, where $F$ is a Borel subset of $\RR$. 
The problems for determining the Hausdorff dimension and capacity of $X^{-1}(F)$  (when
$F \subset \RR$ is fixed) were studied by Hawkes \cite{Hawkes71b, H98} for  strictly
$\al$-stable L\'evy processes in $\RR$, and by Khoshnevisan and Xiao \cite{KX05} for
general L\'evy processes. Their methods are based on potential theory of L\'evy processes.
See Taylor \cite{Tay86} and Xiao \cite{X2004} for further information on fractal properties
of L\'evy processes, where several interesting open problems have remained unsolved.

Motivated by the research in \cite{Hawkes71b, H98,KX05}, an interesting question
would be to provide a dimension formula for $X^{-1}(F)$ that holds simultaneously for
all Borel sets $F \subset \R$. Such a stronger statement, if it is proved to hold, is
customarily referred to as a uniform dimension result. The uniform dimension result 
has a wide applicability because it allows $F$ to be random and dependent of the 
sample paths of $X$. For example, it allows us to compute $\dimh X^{-1}(F)$ when 
$F = X(E)$, where $E \subset [0, \infty)$ is a Borel set. More interestingly, the uniform 
dimension result is useful for 
the multifractal analysis of stochastic processes.  
For instance, \cite{HT,MA,SY} studied 
the multifractal structures of the local times at 0, denoted by $\{L(0, t), t \ge 0\}$, 
of a L\'evy process $X = \{X(t), t\ge 0\}$ that hits points. See Section 4 for the definition 
and more information of local times of L\'evy processes. An open problem is to investigate the 
multifractal structures of the local time processes $\{L(x, t), x \in \R \}$ (when $t$ is fixed) 
or  $\{L(x, t), t \ge 0, x \in \R \}$. For any result on the fractal dimension of the set $F$ of points 
$x$ where $L(x, t)$ has certain (fast or slow) oscillation behavior, one can use the uniform 
dimension result for the inverse images to derive the fractal dimension of the corresponding 
set of times  $X^{-1}(F)$. This is worthy to pursue, but is beyond the scope of this paper. 
 

The uniform dimension result for the level set $X^{-1}(x)$ is also fundamental in the geometric 
construction of the local times. Indeed,  Barlow, Perkins and Taylor \cite{BPT} showed that a 
class of L\'evy processes $\{X(t), t\ge 0\}$ with local times $\{L(x, t), x\in \R, t\ge 0\}$ satisfies
\begin{align*}
\P\left( L (x, t) =  \mathcal H^{\phi}([0,t]\cap \{s: X(s)=x\}) \mbox{ for all } x\in\R, t\ge 0\right)=1,
\end{align*}
where $\phi$ is a sort of gauge function of the level sets. In other words, the local times of $X$ can be 
obtained level-wise by computing certain Hausdorff measure of the level sets. Such a construction does 
not make sense if the uniform Hausdorff dimension result for the inverse images does not hold.

This paper is concerned with the uniform Hausdorff and packing dimension results for
the inverse images of a symmetric L\'evy process, and continues 
the recent investigation of Song, Xiao, and Yang \cite{SXY}, who proved a uniform Hausdorff 
dimension result for the inverse images of an $\al$-stable L\'evy process with $1<\al<2$, and the
classical result of Kaufman \cite{K} for Brownian motion. Let us recall their results.

\begin{thm}[{\cite{K,SXY}}] Let $X = \{X(t), t\ge 0\}$ be a strictly $\al$-stable L\'evy process 
with $1<\al\le 2$. For any $x\in\RR$,
\begin{align*}
\mathbb{P}^x\left(\dimh X^{-1}(F)= 1-\frac{1}{\alpha} +\frac{\dimh F}
{\alpha}
\mbox{ for all Borel sets } F \subseteq \R \right)=1.
\end{align*}
\end{thm}

Our goal in the present paper is twofold.
First, the proof of \cite{SXY} relies on the exact scaling property
satisfied by strictly stable  L\'evy  processes. We intend to show that
certain weak asymptotic behavior of the characteristic exponent $\psi$ (see below) 
suffices to derive the uniform dimension result and such asymptotic behavior
holds for many interesting examples. Second, we consider not only Hausdorff
dimension but also packing dimension, thus provide two ``dual" descriptions for 
the fractal behavior of the inverse images of L\'evy processes.

Let $X=\{X(t), t\ge 0, \P^x\}$ be a real-valued L\'evy process with characteristic (or L\'evy) 
exponent $\psi$, that is, $\E^0[e^{i\la X(t)}]=e^{-t\psi(\la)}$.
Throughout the paper, $X$ is assumed to be symmetric, that is, $X$ and
$-X$ have the same distribution under $\P^0$. 
Consequently, $\psi$ is
real-valued and has the L\'evy-Khintchine representation
$\psi(\la)=\frac12 A \la^2 + 2\int_0^\infty \left(1-\cos(x\la)\right) \nu(dx)$,
where $A\ge 0$ is the variance parameter of the Gaussian part of $X$ and
$\nu$ is called the L\'evy measure satisfying $\nu(\{0\})=0$ and  the integrability condition
$\int_{\R\setminus\{0\}} (1\wedge x^2) \nu(dx)<\infty$.
We suppose that $X$ is a pure-jump L\'evy process, i.e., $A=0$.  

Recall that a Borel measure
$\mu$ on $\RR$ is called unimodal with mode $a$ if $\mu=c\de_a+f(x)dx$
with $c\ge 0$ and $f$ increasing on $(-\infty,a)$, decreasing on $(a,\infty)$.
Unimodality is a time-dependent property for general L\'evy processes.
However, in the symmetric case, it is known that the distribution of
$X(t)$ is unimodal for all $t>0$ if and only if the L\'evy measure
$\nu$ of $X$ is unimodal \cite[p.~488]{W}. In such case, it makes
sense to say that $X$ is a unimodal process.  We refer to \cite{Sato}
for systematic accounts on L\'evy processes.

In order to state our main result, we recall the weak
scaling conditions introduced in \cite{BGR}. 

\begin{defn}\label{def:BGR} 
We say that  $\psi: [0,\infty)\to[0,\infty)$ satisfies the weak lower
scaling condition at infinity if
there exist constants $\underline{\alpha}\in \R$, $\underline{\theta}\geq 0$,
and $\underline{c}\in (0,1]$ such that
$$
\psi(\lam\theta)\geq \underline{c} \lam^{\underline{\alpha}}\psi(\theta),
\quad \lam\geq 1, \theta>\underline{\theta},
$$
and write $\psi\in \ls(\underline{\alpha},\underline{\theta},\underline{c})$.
We say $\psi$ satisfies the weak upper scaling condition at infinity  if there exist
$\overline{\alpha}\in \R$, $\overline{\theta}\geq 0$, and $\overline{C}\in[1,\infty)$
such that
$$
\psi(\lam\theta)\leq \overline{C} \lam^{\overline{\alpha}}\psi(\theta),
\quad \lam\geq 1, \theta>\overline{\theta},
$$
and write $\psi\in \us(\overline{\alpha}, \overline{\theta}, \overline{C})$.

If $\underline{\theta}=0$ we say that $\psi$ satisfies the global weak
lower scaling condition and the $\overline{\theta}=0$
case is called the global weak upper scaling condition.
\end{defn}

It is clear that  $\underline{\alpha} \le \overline{\alpha}$ in Definition 
\ref{def:BGR}. The following Theorems \ref{Thm:H} and  \ref{Thm:P} are the main 
results of this paper.  Theorem \ref{Thm:H} extends \cite[Th.~1.1]{SXY}, and 
Theorem \ref{Thm:P} is new even for Brownian motion. 
We use $\dimh$ and $\dimp$ to denote the
Hausdorff dimension and the packing dimension of a set, respectively.

\begin{thm} \label{Thm:H}
\begin{enumerate}[i)]
\item\label{eqn:Thm:H} Suppose that $X$ is a unimodal symmetric pure jump
L\'evy process in $\RR$ with the L\'evy exponent
$\psi\in \ls(\al , 0, \un c)\cap \us(\al, \ov\theta, \ov C)$
with $1<\al <2$ and some constants $\ov\theta>0$, $\un c$, and $\ov C$.
Then, we have for all $x\in\RR$
\beq\label{eqn:mainub}
\P^{x}\left(\dimh X^{-1}(F)\leq 1-\frac{1}{\alpha}
+\frac{\dimh F}{\alpha} \mbox{ for all Borel sets } F\subseteq \R\right)=1.
\eeq
\item Suppose that $\psi\in \ls(\al , \un\theta, \un c)\cap \us(\al, \overline{\theta}, \ov C)
$ with $\al \in (1,2]$ and $\un\theta, \overline{\theta}>0$.
Then, we have
\beq\label{eqn:mainlb}
\P^{x}\left(\dimh X^{-1}(F)\geq  1-\frac{1}{\alpha}+\frac{\dimh F}{\alpha}
\mbox{ for all Borel sets } F\subseteq \R\right)=1.
\eeq
\item Suppose that $\psi\in \ls(\al , 0, \un c)\cap \us(\al, \overline{\theta}, \ov C)$ with $\al \in (1,2)$ 
and $\overline{\theta}>0$. Then, we have
\beq\label{eqn:main-equal}
\P^{x}\left(\dimh X^{-1}(F)= 1-\frac{1}{\alpha}+\frac{\dimh F}{\alpha}
\mbox{ for all Borel sets } F\subseteq \R\right)=1.
\eeq
\end{enumerate}
\end{thm}

\begin{thm}\label{Thm:P}
\begin{enumerate}[i)]
\item Suppose that the condition of the first part of Theorem \ref{Thm:H} holds.
Then, we have for all $x\in\RR$
\beq\label{eqn:mainubp}
\P^{x}\left(\dimp X^{-1}(F)\leq 1-\frac{1}{\alpha}+\frac{\dimp F}
{\alpha} \mbox{ for all Borel sets } F\subseteq \R\right)=1.
\eeq
\item
Suppose that the condition of Part  ii) of Theorem \ref{Thm:H} holds.
Then, we have for all $x\in\R$
\beq\label{eqn:mainlbp}
\P^{x}\left(\dimp X^{-1}(F)= 1-\frac{1}{\alpha}+\frac{\dimp F}{\alpha}
\mbox{ for all Borel sets } F\subseteq \R
\text{ with } \dimp F=\dimh F
\right)=1.
\eeq
\end{enumerate}
\end{thm}

\begin{remark}\label{r:1.6}
\begin{enumerate}
\item 
The condition that $\dimp F=\dimh F$ in Theorem \ref{Thm:P} is a regularity condition 
on the set $F$ and is technical in nature. 
Even though such condition is satisfied by many
fractal sets, it is natural to ask whether one may remove it. In the case of Brownian motion, 
an affirmative answer can be proved by applying its uniform modulus of continuity 
and the asymptotic property of its local times. 
However, in the general L\'evy processes, we have
not been able to do so due to a difficulty caused by the jumps of $X$.

\item Item ii) of Theorem \ref{Thm:P} follows from item ii) of Theorem
\ref{Thm:H}. Indeed, $\dimp X^{-1}(F)\ge \dimh X^{-1}(F) \ge 1-\frac{1}
{\al}+\frac{\dimh F}{\al}= 1-\frac{1}{\al}+\frac{\dimp F}{\al}$ under
the regularity condition on $F$.

\item As explained in \cite{SXY} when $0<\alpha<1$ the uniform Hausdorff
dimension estimate for the inverse images for symmetric stable processes
can not be true. Therefore one can not expect  Theorem \ref{Thm:H} to hold
when $0<\alpha<1$. The case for $\alpha=1$ is still open even for the
Cauchy process.
\item It is an interesting question to study the dimensions of $X^{-1}(F)$ when 
the upper and lower scaling indices of $\psi$ are different (i.e., $\underline{\alpha} <\overline{\alpha}$). 
In this case, it is possible to extend Lemma \ref{l:cov_pr} and Theorem \ref{thm:HolderLC} 
so that one can derive upper and lower bounds for the Hausdorff and packing dimensions of $X^{-1}(F)$ in 
terms of $\underline{\alpha}$, $\overline{\alpha}$, and the dimensions of $F$.  However, since $F$ 
may vary arbitrarily, there is no hope to obtain equalities in general. Hence we have chosen to state 
Theorems \ref{Thm:H} and  \ref{Thm:P} to give explicit formulae for $\dimh X^{-1}(F)$ and 
$\dimp X^{-1}(F)$.

\item Non-uniform dimension results on the inverse images of L\'evy-type 
processes have been obtained in \cite{KS15},  it would be interesting to obtain uniform 
dimension results for these processes. 
\end{enumerate}
\end{remark}

The general strategy for proving Theorems \ref{Thm:H} and \ref{Thm:P} is
similar to that of \cite{SXY}. To show the uniform (in set $F$) upper bounds
in \eqref{eqn:mainub} and \eqref{eqn:mainubp}, we will prove a covering principle
for the inverse images $X^{-1}(F)$ by applying the recent contribution
of \cite{GR} on the hitting times of a class of L\'evy processes.
On the other hand, for proving the lower bound in \eqref{eqn:mainlb}, we
investigate regularity properties of the local time $L(x, \cdot)$ of $X$, which
can be extended to a Borel measure supported by the level sets of $X$. This
allows us to construct a family of random measures carried by the inverse image
$X^{-1}(F)$ and to establish the desired uniform lower bound by using Frostman's lemma.

This paper is organized as follows.  We present preliminary material
in Section 2.  Theorem \ref{Thm:H} and Theorem \ref{Thm:P} are
proved in Section 3 and 4, respectively. 
Some examples are given in Section 5.  Throughout the paper,  for $f,g:
\RR\to\RR$,  we write $f\asymp g$ if the ratio of the two functions is
bounded from above and from below by some positive finite constants.
Universal constants are denoted by $c,C$ which may differ from line
to line.  
Specific constants are denoted by $c_1,c_2, K_1, K_2$, etc.
Denote by $\E^x$ the expectation with respect to $\P^x$ and for
simplicity, write $\P=\P^0$ and $\E=\E^0$.

\section{Preliminaries}

\subsection{Fractal dimensions}

For definition of Hausdorff dimension, we refer to the monograph of Falconer \cite{Fal}. In the following,
we recall from \cite{Fal} the definition of packing measure and dimension.

Let $s > 0$ be a constant. For any $F \subseteq \R^d$, the \textit{s-dimensional packing measure} of
$F$ is defined by
$$
\mathcal{P}^{s}(F)=\inf\left\lbrace\sum_{i=1}^{\infty}\mathcal{P}_{0}^{s}(F_{i}):\, F\subset
\bigcup_{i=1}^{\infty}F_{i}\right\rbrace,
$$
where $\mathcal{P}_{0}^{s}(F)=\downarrow\lim_{\delta\rightarrow 0}\mathcal{P}_{\delta}^{s}(F)$ and
$$
\mathcal{P}_{\delta}^{s}(F)=\sup\left\lbrace \sum_{i=1}^{\infty}(2 r_{i})^{s} \right\rbrace,
$$
where the supremum is taken over all collections $\{B_i\}$ of disjoint balls of radii $r_i$ at most $\de$
with centers in $F$. The packing dimension of $F$ is defined as
$$
\dimp  F=\sup\{s\geq 0 : {\mathcal P}^{s}(F)=\infty \}=\inf \{s\geq 0 : \mathcal{P}^{s}(F) =0\}.
$$
It can be verified that the packing dimension is stable under countable union in the sense that
\beq\label{eqn:pstable}
\dimp \bigg(\bigcup_{i=1}^{\infty}F_{i}\bigg)=\sup_{i}\dimp F_{i}.
\eeq

For any bounded set $F \subset \R^d$, let $N_{\delta}(F)$ be the smallest number of sets of
diameter at most $\delta$ that covers $F$. 
Then, the upper box dimension of $F$ is defined by
$$
\dimb F =\limsup_{\delta\to 0}\frac{\ln N_{\delta}(F)}{-\ln \delta}.
$$

Note that for any set $F\subset \R^{d}$ (see \cite[Equation (3.27)]{Fal})
\begin{equation}\label{e:dim_comparison}
\dimh F\leq \dimp F\leq \dimb F.
\end{equation}

Moreover,  packing and upper box dimensions are related by the following
regularization procedure (cf. \cite{Fal}):
\beq\label{eqn:regularization}
\dimp F=\inf \bigg\{\sup_{n} \dimb F_{n} :\, F\subset \bigcup_{i=1}^{\infty}F_{n} \bigg\},
\eeq
where the infimum is taken over all $\{F_{n}\}$ such that $F \subset \bigcup_{i=1}^{\infty}F_{n}$.

\subsection{Weak scaling condition}

Now we focus on the weak scaling conditions.  We start with some notation. 
Let $\phi:[0,\infty)\rightarrow [0,\infty)$
be a continuous and nondecreasing function. We define its
generalized inverse as
$$
\phi^{-1}(u)=\inf\{s\geq 0 : \phi(s)\geq u\}, \quad u \in [0,\phi(\infty)].
$$
Note that $\phi^{-1}$ is nondecreasing and it is left continuous and has a right limit at every point. 
Also it is easy to see
$\phi^{-1}(\phi(u))\leq u \hbox{ and } \phi(\phi^{-1}(u))=u. $

If $\phi$ is not necessarily nondecreasing, we define the maximal function
 $\phi^{*}$ as
$$
\phi^{*}(y)=\sup \{\phi(x) : 0\leq x\leq y\}.
$$
Motivated by the relation
$
\inf\{s : \phi(s)\geq u\}=\inf\{s : \phi^{*}(s)\geq u\},
$
we define
$
\phi^{-1}:=(\phi^{*})^{-1}.
$
Note that we still have $\phi^{-1}(\phi(u))\leq u$ and $\phi(\phi^{-1})(u)=u$.

Let us review some consequences of the weak scaling conditions.  Each point is used in
a specific stage of the proof.

\begin{itemize}
\item The global weak lower scaling condition implies an asymptotic 
result for the hitting times of $X$ on a compact interval, see Theorem 
\ref{t:gr} below.
\item  We need in the proof that a re-scaled version of $X$, denoted 
by $Y_b= \{b X(b^{-\al}t), t\ge 0\}$ for the moment, satisfies 
$\P^0(Y_b(1)\in [2,3])>c$ uniformly for all large $b$.  This property 
is guaranteed whenever the characteristic exponent of $X$ satisfies 
$\psi \in\ls(\al, \theta, \un c) \cap \us(\al, \theta, \ov C)$ 
with $0< \al<2$ and that $X$ is unimodal, as suggested 
by Theorem \ref{t:bgr} below.
\end{itemize}

Denote by $T_{A}=\inf\{t>0 : X(t)\in A\}$ the first hitting time of 
$A \subset \R$ by $X$. 
The following theorem is taken from \cite{GR} which provides a sharp tail 
probability estimate for the first hitting time of an interval.
\begin{thm}(\cite[Th.~5.5]{GR})\label{t:gr}
Suppose that $\psi\in \ls(\underline{\alpha},0,\underline{c})$ with 
$\underline{\alpha}>1$. 
Then for any $x \in \R$ with  $|x|>R$,
$$
\P^{x}(T_{[-R,R]} > t)\asymp\frac{V(|x|-R)K(|x|)}{V(|x|)t
\psi^{-1}(1/t)}\wedge 1, \quad t>1/\psi^{*}(1/R),
$$
where the comparability constant depends only on the scaling 
characteristics,
\begin{align*}
K(x)=\frac 1 \pi \int_0^\infty (1-\cos xs)\frac 1 {\psi(s)}ds,
\end{align*}
and $V(x)$, $x\ge 0$, is the potential measure of the interval 
$[0,x]$ which satisfies
\begin{align*}
V(r)\asymp \frac{1}{\sqrt{\psi^*(1/r)}}.
\end{align*}
\end{thm}

Recall the following result on the lower bound for the transition 
density of unimodal symmetric L\'evy processes with the weak scaling 
properties.
\begin{thm}(\cite[Th.~21]{BGR})\label{t:bgr}
Let $X$ be a unimodal symmetric L\'evy process in $\R$ with 
characteristic exponent $\psi$. If $\psi\in \ls(\underline{\alpha},
\theta,\underline{c})\cap \us (\overline{\alpha}, \theta,
\overline{C})$ with $0< \un\al, \ov\al<2$, then there exist constants 
$c^{*}$ and $r_{0}$ such that the  transition density $p(t,x)$ satisfies
$$
p(t,x)\geq c^{*}\left(\psi^{-1}(1/t) \wedge \frac{t\psi^{*}(1/|x|)}
{|x|}\right) \ \ \text{ if } t>0,
\  t\psi^{*}(\theta/r_{0})<1, \text{ and } \  |x|<r_{0}/\theta.
$$
\end{thm}

\section{Proof for the upper bounds}

Let us start with the upper bound.
We establish a covering principle for the inverse images of $X$, 
which is a reminiscence of \cite[Lemma 2.2]{SXY}. Let $\mathcal{U}_{n}$ 
be any partition of $\R$ with intervals of length $2^{-n}$ and 
$\mathcal{D}_{n} (\ov{\alpha})$ be any partition of 
$[0,\infty)$ with length $2^{-n\ov{\alpha}}$. 

\begin{lemma}\label{l:cov_pr}
Suppose that $\psi\in \ls(\alpha,0,\underline{c})
\cap\us(\alpha,\theta,\ov C)$ with $1
< \alpha <2$ 
and for some constant $\theta>0$.
Let $\de>\al-1$ and $T>0$. 
$\P^x$-a.s. for all integer $n$ 
sufficiently large and each $U\in\mathcal U_n$, $X^{-1}(U)
\cap [0,T)$ can be covered by $2\cdot 2^{n\de}$ intervals 
from $\mathcal D_n(\al)$.
\end{lemma}
\begin{proof}
{\it Step 1.}  For a fixed $U\in\mathcal U_n$, write $U=
(z-\frac{2^{-n}}{2}, z+\frac{2^{-n}}{2})$ for some $z\in\R$.
Define $\tau_0=0$ and for all integer $k\ge 1$, $$
\tau_{k}=\inf\{s>\tau_{k-1} +2^{-n\alpha}:\, 
X(s)\in U\}
$$
with the convention that $\inf\emptyset=\infty$. From the 
fact that 
$X^{-1}(U)\subset \bigcup_{i=0}^{\infty} [\tau_i,
\tau_i+2^{-n\al}]$, 
we note that for any $T>0$,
\begin{align*}
\{X^{-1}(U)\cap[0,T) \mbox{ cannot be covered by } k 
\mbox{ intervals of length } 2^{-n\al}\}\subset \{\tau_k<T\}.
\end{align*}
Due to the right continuity of the sample paths of $X$, we 
observe that $X(\tau_{k-1})$ belongs to the closure of $U$ 
as $\tau_{k-1}<T$. By the strong Markov property, we obtain
\begin{align*}
\P^x(\tau_k<T) &= \P^{x}(\tau_{k}< T | \tau_{k-1}< T) 
\P^x(\tau_{k-1}<T) \\
&\le \sup_{y\in\ov U}\P^y\left( \inf_{2^{-n\alpha} 
\leq t\leq T}|X(s)-z|\leq \frac{2^{-n}}{2}\right)
\P^x(\tau_{k-1}<T) \\
&\le  \sup_{y\in\ov U}\P^y\left( \inf_{2^{-n\alpha} 
\leq t\leq T}|X(s)-y|\leq 2^{-n}\right)
\P^x(\tau_{k-1}<T).
\end{align*}
Define a sequence of intermediate processes 
$Y_n=\{2^nX(2^{-n\al}t), t\ge 0\}$. 
It follows from the spatial homogeneity of L\'evy processes that for any $y\in\R$,
\begin{align*}
\P^y\left( \inf_{2^{-n\alpha} \leq t\leq T}|X(s)-y|
\leq 2^{-n}\right)
= \P\left(\inf_{1\le t\le T2^{n\al}} |Y_n(t)|\le 1\right):=p_n.
\end{align*}
By induction, we obtain
\begin{align*}
\P^x(\tau_k<T)\le p_n^k.
\end{align*}

{\it Step 2.} 
Now we intend to prove that $1-p_n\ge C_T 2^{-n(\al-1)}$, 
thus find an upper bound for $p_n$. Let $T^n_A$ be the hitting time of 
$Y_n$ to any interval $A$. Considering the complement of the event
associated with $p_n$, we have by the independence of increments that
\begin{align*}
1-p_n&\ge \P\Big(2\le |Y_n(1)|\le 3,\ \  \inf\{t\ge 1: Y_n(t)-Y_n(1)
\in[-4,-1]\} \ge T2^{n\al}+1\Big) \\
&= \P \big(2\le |Y_n(1)|\le 3\big)\,  \P\big(T^n_{[-4,-1]}>T2^{n\al}\big).
\end{align*}

We proceed by looking for lower bounds for these events on $Y_n$. First,
\begin{align*}
\P(2\le |Y_n(1)|\le 3) = \P( 2\cdot 2^{-n} \le |X(2^{-n\al})|
\le 3\cdot 2^{-n})
\ge \int_{2\cdot 2^{-n}}^{3\cdot 2^{-n}} p(2^{-n\al},x)dx.
\end{align*}
For $n$ sufficiently large, $2^{-n\al}\psi^*(\theta/r_0)<1$ and 
$3\cdot 2^{-n}<r_0/\theta$.
Since $\psi\in \ls(\alpha,0,\underline{c})
\cap\us(\alpha,\theta,\ov C)$ with $1
< \alpha <2$, it follows from \cite[Remark 4]{BGR} we have $\psi^{-1}\in \us(1/\alpha, 0, \underline{c}^{-1/\alpha})\cap \ls(1/\alpha, \psi (\ov\theta), \ov{C}^{-1/\alpha})$.
Note that for $x\in [2\cdot 2^{-n}, 3\cdot 2^{-n}]$ and $t=2^{-n\alpha}$ 
we have $\psi^{-1}(2^{n\alpha})\asymp 2^{n}$. Since $\psi\in \ls(\alpha,0,\underline{c})
\cap\us(\alpha,\theta,\ov C)$, together with \cite[Proposition 2]{BGR}, we have $\frac{t\psi^{*}(1/|x|)}{|x|}
\asymp 2^{-n(\alpha-1)}\psi^{*}(2^{n})\asymp 2^{n}$. This shows that
$\psi^{-1}(1/t)$ is comparable to $\frac{t\psi^{*}(1/|x|)}{|x|}$ for $t=2^{-n\alpha}$ and $x\in [2\cdot 2^{-n}, 3\cdot 2^{-n}]$.
It follows from Theorem \ref{t:bgr} and the weak lower scaling property of 
$\psi$ that  $p(2^{-n\al},x)\ge \frac{2^{-n\al}\psi^*(1/x)}{x}\ge c 2^n$ 
on $x\in [2\cdot 2^{-n},
3\cdot 2^{-n}]$.  
Therefore, 
$$
\P(2\le |Y_n(1)|\le 3)\ge c>0
$$ 
uniformly for all sufficiently large $n$.

Second, observe that $Y_n$ under $\P$ has the characteristic 
exponent $\psi_n(\la)=2^{-n\al} \psi(2^n\la)$, and it is easy 
to check that $\psi_n\in \ls(\al,0,\un c)$ with the same scaling
characteristics as those of $\psi$. Hence,  by applying the  spatial homogeneity of 
$X$ and Theorem 
\ref{t:gr} with $x=\frac{5}{2}$,
$R=\frac{3}{2}$, $t=T 2^{n\al}$, we arrive at
\begin{align*}
\P\Big(T^n_{[-4,-1]}>T2^{n\ov\al} \Big)
= \P^{5/2}\left(T_{[-3/2,3/2]}^{n}>T2^{n\alpha}\right)
\asymp \frac{1}
{T2^{n\al}\psi^{-1}_n(T^{-1}2^{-n\al})}.
\end{align*}
Note that if $f(x)=a h(bx)$, then $f^{-1}(x)=b^{-1}h^{-1}(a^{-1}x)$.
Applying this relation to $f=\psi_n$, $h=\psi$, $a=2^{-n\al}$, 
$b=2^n$,  we obtain $\psi_n^{-1}(T^{-1} 2^{-n\al}) = 2^{-n}\psi^{-1}
(2^{n\al}\cdot T^{-1}2^{-n\al}) =2^{-n}\psi^{-1}(T^{-1})$.
Therefore,
\begin{align*}
\P \Big(T^n_{[-4,-1]}>T2^{n\al} \Big)\asymp C_T 2^{-n(\al-1)},
\end{align*}
as desired.

{\it Step 3.}  Define the event $A_{n}^{\delta}(T)$ by
\[
\begin{split}
A_{n}^{\delta}(T) & = \Big\{\exists U\in \mathcal{U}_{n}\cap 
[-K,K] \text{ s.t. } X^{-1}(U)\cap [0,T] \text{ cannot be
} \\
& \qquad \qquad  \qquad \qquad  \qquad \qquad   
\text{ covered by } \ 2^{n\delta} \text{ intervals of length }
 2^{-n\alpha}\Big\},
\end{split}
\]
where $U\in \mathcal{U}_{n}\cap [-K,K]$ means that $U \in 
\mathcal{U}_n$ and $U\subset [-K,K]$. 
We have for $\delta>\al -1$,
\begin{eqnarray*}
&&\sum_{n=1}^{\infty}\P^{x}(A_{n}^{\delta}(T)) \leq \sum_{n=1}^{\infty}
2K2^{n}(p_{n})^{2^{n\delta}}\leq 2K\sum_{n=1}^{\infty}2^{n}\big(1-c_T 2^{-n(\alpha-1)}\big)^{2^{n\delta}}\\
&=&2K\sum_{n=1}^{\infty}\exp\left(n(\ln 2)-C_T 2^{n(\de-\alpha+1)}\right)<\infty.
\end{eqnarray*}
The conclusion for all $U\subset[-K,K]\cap \mathcal U_n$ follows from the Borel-Cantelli Lemma.
Letting $K\to\infty$ completes the proof.
\end{proof}

\noindent\textbf{Proof of Theorem \ref{Thm:H} i)}

Suppose that $\psi\in \ls(\al , 0, \un c)\cap \us(\al, 
\theta, \ov C))$. 
The proof is similar to the proof of \cite[Theorem 1.1]{SXY} and we provide the details for the reader's convenience. 
Let $F$ be any Borel set and take $\gamma>\dimh F$ and $\delta>\alpha-1$.
There exists a sequence of intervals $\{U_{i}\}$ of length $2^{-n_{i}}$ such that $F\subset \cup_{i=1}^{\infty}U_{i}$ and $\sum_{i=1}^{\infty}2^{-n_{i}\gamma}<1$. 
For any $T>0$ it follows from Lemma \ref{l:cov_pr}, we have
$$
X^{-1}(F)\cap [0,T]\subset \bigcup_{i=1}^{\infty}\bigcup_{k=1}^{2\cdot 2^{n_{i}\delta}}I_{i,k},
$$
where $I_{i,k}$ are in $\mathcal{D}_{n_{i}}(\alpha)$. This implies
$$
\sum_{i=1}^{\infty}\sum_{k=1}^{2\cdot 2^{n_{i}\delta}}\text{diam}(I_{i,k})^{\frac{\gamma+\delta}{\alpha}}
=2\cdot 2^{n_{i}\delta}\sum_{i=1}^{\infty}(2^{-n_{i}\alpha})^{\frac{\gamma+\delta}{\alpha}}
=2\sum_{i=1}^{\infty}2^{-n_{i}\gamma}<\infty.
$$
Hence, 
$$
\dimh \left(X^{-1}(F)\cap [0,T]\right)\leq \frac{\gamma +\delta}{\alpha}.
$$
Letting 
$\gamma \downarrow \dimh F$, $\delta\downarrow \alpha-1$ 
and $T\uparrow\infty$ (all along rational numbers) gives
$$
\dimh X^{-1}(F)\leq \frac{\dimh F +\alpha -1}{\alpha}.
$$
\qed

Now we establish the upper bound for the packing dimension.
By the second point of Remark \ref{r:1.6}, it suffices to 
prove the upper bound of Theorem \ref{Thm:P}.

\noindent\textbf{Proof of Theorem \ref{Thm:P} i)}

We will first prove that for any given $T>0$,
\beq\label{eqn:boxstep1}
\P^{x}\left(\dimb (X^{-1}(F)\cap [0,T]) \leq 1 -\frac{1}
{\alpha} +\frac{\dimb(F)}{\alpha}\right)=1.
\eeq
Let $\theta>\dimb(F)$ so that
$$
N_{2^{-n}}(F)2^{-n\theta}\to 0 \text{ as } n\to\infty.
$$
Then, there exist intervals $\{U_{i}\}$ with length $2^{-n}$ 
such that $F\subset \cup_{i=1}^{N_{2^{-n}}(F)}U_{i}$.
Let $\delta >\alpha-1$.
It follows from Lemma \ref{l:cov_pr} that for all $n$ sufficiently 
large,
$$
X^{-1}(F)\cap [0,T] \subseteq \bigcup_{i=1}^{N_{2^{-n}}(F)} 
\Big(X^{-1}(U_{i})\cap [0,T] \Big) \subseteq 
\bigcup_{i=1}^{N_{2^{-n}}(F)}\bigcup_{k=1}^{2\cdot 2^{n\delta}} V_{ik},
$$
where $|V_{ik}|=2^{-n\alpha}$.
Hence, 
$$
N_{2^{-n\alpha}}(X^{-1}(F)\cap [0,T])
\leq 2^{1+n\delta}\, N_{2^{-n}}(F).
$$
Let $d=\frac{\theta+\delta}{\alpha}$. 
Then we have
$$
N_{2^{-n\alpha}}(X^{-1}(F)\cap [0,T])(2^{-n\alpha})^{d}\leq 2 N_{2^{-n}}(F)2^{-n\theta}\to 0 \text{ as } n\to\infty
$$
and this implies that $\dimb X^{-1}(F)\cap [0,T] \leq d$ a.s.
Letting $\theta \downarrow \dimb (F)$ and $\delta \downarrow  \alpha-1$ prove \eqref{eqn:boxstep1}.

Now take any cover of $F\subset \bigcup F_{n}$. It follows from 
\eqref{eqn:pstable}, \eqref{e:dim_comparison}, and \eqref{eqn:boxstep1} 
that
\begin{equation}\label{eqn:inter}
\begin{split}
\dimp \big(X^{-1}(F)\cap [0,T]\big)&\leq\dimp \bigcup_{n=1}^{\infty}\big(X^{-1}(F_{n})\cap [0,T]\big)
= \sup_{n}\dimp \big(X^{-1}(F_{n})\cap [0,T]\big)\\
&\leq\sup_{n}\dimb \big(X^{-1}(F_{n})\cap [0,T]\big)\leq 1-\frac{1}{\alpha}+\frac{\dimb F_{n}}{\alpha}.
\end{split}
\end{equation}
Since the left hand side of \eqref{eqn:inter} does not depend on $\{F_{n}\}$, an application of \eqref{eqn:regularization} yields
$$
\P^{x}\left(\dimp \left(X^{-1}(F)\cap [0,T]\right)\leq 1-\frac{1}{\alpha}+\frac{\dimp F}{\alpha}\right)=1.
$$
Finally, we let $T\to \infty$ and this proves \eqref{eqn:mainubp}.

\section{Proof for the lower bound}

We move to prove the lower bound in \eqref{eqn:mainlb}. We first 
establish a uniform H\"older-type condition for the local times of
such processes by using the method of moments which is
similar to Khoshnevisan et al \cite{KXZ} or Xiao \cite{X}.

We first recall the \textit{lower index} $\beta^{\text{low}}$
of an arbitrary L\'evy process introduced by Blumenthal and 
Getoor \cite{BG61}, 
$$
\beta^{\text{low}}=\sup \Big\{\gamma\geq 0 : \lim_{|\xi|\to\infty} 
\|\xi\|^{-\gamma} \text{Re}\psi(\xi)=\infty \Big\},
$$
where $\text{Re}\psi(\xi)$ represents the real part of its 
argument. Since the process $X$ is symmetric, we have 
$\text{Re}\psi(\xi)=\psi(\xi)$. Also notice that when $\psi
\in \ls(\underline{\alpha},\underline{\theta}, \underline{c})$ 
we have $\beta^{\text{low}}\geq \underline{\alpha}$.

Let $A\subset [0,\infty)$ be a Borel set and define the 
occupation measure $\mu_{A}(\cdot)$ by
$$
\mu_{A}(\cdot)=m\big(\{t\in A : X(t)\in \cdot\} \big),
$$
where $m(\cdot)$ is the Lebesgue measure in $\R$.
For any Borel set $A\subset [0,\infty)$, if $ \mu_{A} \ll m$,  then we define a local time 
of $X$ on $A$ by    
$$
L(x,A):=\frac{d\mu_{A}}{dm}(x).
$$
In this paper, we only consider $A = [0, \infty)$ and, in this case, Hawkes \cite{H86} 
showed that a necessary and sufficient condition for 
the existence of local times of a  L\'evy process $X$ with exponent 
$\psi$ is ${\rm Re}\big( \frac 1 {1 + \psi(\xi)}\big) \in L^1(\R)$.
We will write $L(x,t)$ for $L(x,[0,t])$. If there is a 
modification of the local time such that it is continuous 
in $(x, t)$, we say that $X$ has a jointly continuous local 
time. Necessary and suffcient conditions for the joint continuity of the local times of 
L\'evy processes have been proved by Barlow and Hawkes \cite{BH}, Barlow \cite{Barlow},
and by  Marcus and Rosen \cite{MR92}  using different methods.

Since $\psi \in \ls(\underline{\alpha},0, \underline{c})$ 
it follows from \cite[Theorem 2.1]{KXZ} that there exists 
a square integrable local time for $X$. We also note that 
under the current setting (with $N=1$) the proof of 
\cite[Theorem 3.2]{KXZ} holds true. Consequently,  
\cite[Equations (3.16) and (3.17)]{KXZ} hold true as well. 
This establishes 
the following estimates for the local time of $X$.

\begin{lemma}{{\cite[Lemma 4.2]{KXZ}}}\label{lemma:local times}
Suppose that $X$ is a symmetric unimodal L\'evy process and 
$\psi\in \ls(\underline{\alpha}, \underline{\theta}, \underline{c})$ 
with $\underline{\alpha}>1$. For any $\gamma \in (0,\, 
\frac{\underline{\alpha}-1}{2})$, there exist constants $b_{1},b_{2}>0$ 
and $0<K_{1}, K_{2}<\infty$ such that for any interval $I=[a,a+h]$, 
$a\geq 0$, $x,y\in \R$, and $u>0$,
$$
\P\left(L(X_{t}+x, I) \geq  h^{1-\frac{1}{\underline{\alpha}}} 
u^{\frac{1}{\underline{\alpha}}}\right)\leq K_{1}e^{-b_{1}u},
$$
and
$$
\P\left(\left|L(X_{t}+x, I)-L(X_{t}+y, I)\right| \geq  h^{1-
\frac{1+\gamma}{\underline{\alpha}}} |x-y|^{\gamma}u^{\frac{1+\gamma}
{\underline{\alpha}}}\right)\leq K_{2}e^{-b_{2}u},
$$
where either $t=0$ or $a$.
\end{lemma}

We need to estimate the local oscillation of the process $X$, which is given in the following lemma.  
It is a direct consequence of a general result for L\'evy-type processes. 

\begin{lemma}[{\cite[Th. 5.1]{BSW}}]\label{l:bsw}
Assume that $X$ is a symmetric L\'evy process. Then, there exists a constant $c$ such that for 
any $t,\,r>0$,
\begin{align*}
\P \Big(\sup_{0\leq s\le t } |X(s)|>r \Big) \le c t \psi^*(1/r).
\end{align*}
\end{lemma}
If $\psi\in \us(\delta, \overline{\theta}, \overline{C})$, we have a versatile version of Lemma \ref{l:bsw}.
\begin{lemma}\label{lemma:fluctuation}
Suppose that $\psi\in \us(\delta, \overline{\theta}, \overline{C})$ with $\de\in(0,2]$.
Then for $r\leq \frac{1}{\overline{\theta}}$ we have
$$
\P\Big(\sup_{0\leq s \leq t}|X(s)|>r \Big)\leq c tr^{-\delta}.
$$
\end{lemma}
\pf
Applying \cite[Proposition 1]{G} to our process $X$, we have
\beq\label{eqn:fluctuation3}
\psi^{*}(|x|)\leq 12 \psi(x), \quad x\in\R.
\eeq
Hence it follows from Lemma \ref{l:bsw}, 
the weak upper scaling condition,  and \eqref{eqn:fluctuation3} that for  
$r\leq \frac{1}{\ov\theta}$
\[
\begin{split}
&\P\bigg(\sup_{s\leq t}|X(s)|>r \bigg)\leq \ov Ct\psi^{*}(1/r)\leq 12\ov C t\psi(1/r)
\leq c\,t\psi(\overline{\theta})
\overline{\theta}^{-\delta}r^{-\delta}.
\end{split}
\]
This proves the lemma. \qed

Now we are ready to prove a uniform H\"older condition for
the local times for $X$. The proof is similar to that of 
\cite[Theorem 4.3]{KXZ}  or \cite[Theorems 1.1 and 1.2]{X} 
with obvious modifications. We provide the details for the 
reader's convenience. For an interval $I\subset [0,\infty)$ 
we define $L^{*}(I)=\sup_{x\in \R}L(x,I)$ to be the 
maximum local time of $X$ on $I$.

\begin{thm}\label{thm:HolderLC}
Suppose that $X$ is a symmetric unimodal L\'evy process in $\R$
and its characteristic exponent $\psi$ satisfies
$$
\psi\in \ls(\alpha,\underline{\theta},
\underline{c})\cap \us(\alpha, \ov{\theta}, \overline{C})
 $$ 
with $\alpha >1$ and for some constants $\underline{c}, \overline{C}>0$. 
Let $L$ be its jointly continuous local time 
and fix $\tau>0$ and $N>0$. Then we have
\beq\label{eqn:Holder continuity1}
\limsup_{r\to 0}\frac{L^{*}([\tau-r, \tau+r])}{r^{1-\frac{1}
{\alpha}}(\ln\ln r^{-1})^{\frac{1}{\alpha}}}<\infty,
\eeq
and
\beq\label{eqn:Holder continuity2}
\limsup_{r\to 0}\sup _{I\subset [0,N], m(I)<r}\frac{L^{*}(I)}
{r^{1-\frac{1}{\alpha}}\ln(r^{-1})^{\frac{1} {\alpha}}}<\infty.
\eeq
\end{thm}
\pf
We first prove \eqref{eqn:Holder continuity1}.
Let $g(r)=r^{1-\frac{1}{\alpha}}(\ln\ln\frac{1}{r})^{\frac{1}{\alpha}}$
and $C_{n}=[s, s +\frac{1}{2^{n}}]$ for any $s\ge 0$. We will prove that
\beq\label{eqn:Holder continuity3}
\limsup_{n\to\infty}\frac{L^{*}(C_{n})}{g(2^{-n})}<\infty.
\eeq
Since $2^{-n/\alpha}n^{\beta}<\theta$ for all but finitely
many $n$'s, it follows from Lemma \ref{lemma:fluctuation}
and the fact $\psi\in \us(\alpha,\theta, \ov C)$ 
that for all sufficiently large $n$, 
$$
\P\Big(\sup_{t\in C_{n}} |X_{t}-X_{s}|>2^{-n/\alpha}
n^{\beta}\Big)\leq 2^{-n}(2^{-n/\alpha}n^{\beta})^{\alpha}
=n^{-\alpha\beta}.
$$
We choose $\beta>\frac{1}{\alpha}$.
Then, the Borel-Cantelli lemma gives that a.s.
$$
\sup_{t\in C_{n}} |X_{t}-X_{s}|\leq 2^{-n/\alpha}n^{\beta}
$$
for all but finitely many $n$'s.

Now let $\theta_{n}:=2^{-n/\alpha}(\ln n)^{-1/\alpha}$.
Define
$$
G_{n}=\{x : |x|\leq 2^{-n/\alpha}n^{\beta}, x=p\theta_{n} 
\text{ for some } p\in \mathbb{Z}  \}.
$$
Note that the number of elements in $G_{n}$ is at most
$\frac{2^{-n/\alpha}n^{\beta}}{\theta_{n}}=n^{\beta}(\ln n)^{1/\alpha}$.
Choose a constant $a_{1}$ so that $b_{1}a_{1}-\beta >1$, where 
$b_{1}$ is a constant from Lemma \ref{lemma:local times}.
Then, Lemma \ref{lemma:local times} implies 
\[
\begin{split}
&\P\left(\max_{x\in G_{n}} L(X_{s}+x, C_{n})\geq a_{1}^{1/\alpha}g(2^{-n})\right)\\
&=\P\left(\max_{x\in G_{n}} L(X_{s}+x, C_{n})\geq (2^{-n})^{1-\frac{1}{\alpha}}(a_{1}\ln\ln 2^{n})^{1/\alpha}\right)\\
&=(\# \text{ of elements in } G_{n}) \times K_{1}e^{-b_{1}a_{1}\ln\ln 2^{n}}\\
&\leq n^{\beta}(\ln n)^{1/\alpha}\times K_{1}e^{-b_{1}a_{1}\ln\ln 2^{n}}
= K_{1}e^{-b_{1}a_{1}\ln\ln 2}(\ln n)^{1/\alpha}n^{-(b_{1}a_{1}-\beta)},
\end{split}
\]
and since $b_{1}a_{1}-\beta >1$ we have
$\sum_{n=1}^{\infty}K_{1}e^{-b_{1}a_{1}\ln\ln 2}(\ln n)^{1/\alpha}
n^{-(b_{1}a_{1}-\beta)}<\infty$.
Hence, again the Borel-Cantelli lemma yields
\beq\label{eqn:Holder1}
\max_{x\in G_{n}} L(X_{s}+x, C_{n})<a_{1}^{1/\alpha}g(2^{-n})
\eeq
for all but finitely many $n$'s.

Choose $\gamma$ so that $\gamma<\frac{\alpha-1}{2}$.
Let
$$
B_{n}=\displaystyle \bigcup_{k=1}^{\infty}
\bigcup_{y_{1}, y_{2}} \Big\{\left| L(X_{s}+y_{1}, C_{n})-
L(X_{s}+y_{2}, C_{n})\right|  \geq (2^{-n})^{1-\frac{1+\gamma}
{\alpha}}|y_{1}-y_{2}|^{\gamma} (a_{2}k\ln n)^{\frac{1+\gamma}{\alpha}}\Big\},
$$
where $y_{1}$ and $y_{2}$ are lattice points in $G_{n}$ 
that satisfy $y_{1}-y_{2}=\theta_{n}2^{-k}$. Note that for 
each $k$ there are $2^{k}$ such
pairs of $y_{1}$ and $y_{2}$.
It follows from Lemma \ref{lemma:local times} that
\[
\begin{split}
\P(B_{n}) &=n^{\beta}(\ln n)^{\frac{1}{\alpha}}\sum_{k=1}^{\infty}
2^{k}K_{2}e^{-b_{2}a_{2}k\ln n}\leq 3K_{2}(\ln n)^{\frac{1}{\alpha}}
n^{-(b_{2}a_{2}-\beta)},
\end{split}
\]
where we have used the fact that
$$
\sum_{k=1}^{\infty}2^{k}n^{-b_{2}a_{2}k}=\sum_{k=1}^{\infty}\left(\frac{2}{n^{b_{2}a_{2}}}\right)^{k}
=\frac{\frac{2}{n^{b_{2}a_{2}}}}{1-\frac{2}{n^{b_{2}a_{2}}}}=\frac{2}{n^{b_{2}a_{2}}-2}\leq 3n^{-b_{2}a_{2}}
$$
for all sufficiently large $n$. 
Hence, by taking $a_{2}$ so 
that $b_{2}a_{2}-\beta>1$ we have by the Borel-Cantelli lemma
$$
\P \Big(\limsup_{n}B_{n}\Big)=0.
$$

Now suppose that $|y|<2^{-\frac{n}{\alpha}}n^{\beta}$.
Then, we can express $y$ as 
$y=\lim_{k\to\infty}y_{k}$, where $y_{0}=x$ and $y_{k}=x+\theta_{n}
\sum_{j=1}^{k}\eps_{j}2^{-j}$. 
Hence, it follows from the triangular 
inequality that on the event $\big\{\limsup_{n}B_{n}\big\}^{c}$ 
and $n$ sufficiently large,
\begin{equation}\label{eqn:Holder2}
\begin{split}
|L(X_{s}+y, C_{n})-L(X_{s}+x, C_{n})| &\leq\sum_{k=1}^{\infty}|L(X_{s}+y_{k}, C_{n})-L(X_{s}+y_{k-1}, C_{n})| \\
&\leq\sum_{k=1}^{\infty}(2^{-n})^{1-\frac{1+\gamma}{\alpha}}|y_{k}-y_{k-1}|^{\gamma}(a_{2}k\ln n)^{\frac{1+\gamma}{\alpha}}\\ 
&\leq\sum_{k=1}^{\infty}(2^{-n})^{1-\frac{1+\gamma}{\alpha}}\left(2^{-\frac{n}{\alpha}}(\ln n)^{-\frac{1}
{\alpha}}2^{-k}\right)^{\gamma}(a_{2}k\ln n)^{\frac{1+\gamma}{\alpha}}\\
&= (2^{-n})^{1-\frac{1}{\alpha}}(\ln n)^{\frac{1}{\alpha}}a_{2}^{\frac{1+
\gamma}{\alpha}}\sum_{k=1}^{\infty}2^{-k\gamma}k^{\frac{1+\gamma}{\alpha}}\leq c\, g(2^{-n})
\end{split}
\end{equation}
for some finite constant $c>0$. Hence \eqref{eqn:Holder continuity3} 
follows from \eqref{eqn:Holder1} and \eqref{eqn:Holder2}.

Next we prove \eqref{eqn:Holder continuity2}.
For simplicity we may and will assume that $I\subset [0,1]$. Let $\mathcal{D}_{n}$ be a collection of $2^{n}$ non-overlapping
intervals in $[0,1]$ of length $\frac{1}{2^{n}}$.
Define $h(r)=r^{1-\frac{1}{\alpha}}(\ln r^{-1})^{\frac{1}{\alpha}}$.
Let $\eta_{n}:=2^{-\frac{n}{\alpha}}n^{-\frac{1}{\alpha}}$ and define
$$
H_{n}:=\{x\in \R:\, |x|\leq n, \ x=\eta_{n}p, \ p\in \mathbb{Z}\}.
$$
Note that the cardinality $\#(H_{n})$ of $H_{n}$ satisfies
$$
\# H_{n}\leq \frac{2n}{\eta_{n}}=\frac{2n}{2^{-\frac{n}{\alpha}}n^{-\frac{1}{\alpha}}}
=2n^{1+\frac{1}{\alpha}}2^{\frac{n}{\alpha}}.
$$
It follows from Lemma \ref{lemma:local times} that
\[\begin{split}
&\P\left(\max_{x\in H_{n}} L(x, B)\geq a_{3}^{\frac{1}{\alpha}}h(2^{-n}) \text{ for some } B\in \mathcal{D}_{n}\right)\\
&\leq \P\left(\max_{x\in H_{n}} L(x, B)\geq (2^{-n})^{1-\frac{1}{\alpha}}(a_{3}n\ln 2)^{\frac{1}{\alpha}}
\text{ for some } B\in \mathcal{D}_{n}\right)\\
&\leq 2n^{1+\frac{1}{\alpha}}2^{\frac{n}{\alpha}} \times 2^n  K_{1}e^{-b_{1}a_{3}n\ln 2}= 2K_{1}
n^{1+\frac{1}{\alpha}}2^{-n(a_{3}b_{1}-\frac{1}{\alpha})}.
\end{split}
\]
Hence, by taking $a_{3}$ so that $a_{3}b_{1}-\frac{1}{\alpha}>0$ we have
\beq\label{eqn:Holder continuity4}
\max_{x\in H_{n}} L(x, B)< a_{3}^{\frac{1}{\alpha}}h(2^{-n}) \text{ for all } B\in\mathcal{D}_{n}
\eeq
for all sufficiently large $n$.

Let
$$
D_{n}=\displaystyle \bigcup_{k=1}^{\infty}
\bigcup_{y, y'} \Big\{\left| L(y, B)-L(y', B)\right|  \geq (2^{-n})^{1-\frac{1+\gamma}{\alpha}}
|y-y'|^{\gamma}(a_{4}k n)^{\frac{1+\gamma}{\alpha}}
\text{ for some } B\in \mathcal{D}_{n}\Big\},
$$
where $y$ and $y'$ are lattice points in $H_n$ that satisfy
 $y-y'=\eta_{n}2^{-k}$.
Hence, by Lemma \ref{lemma:local times} we have
\[
\begin{split}
\P(D_{n})&\leq 2n^{1+\frac{1}{\alpha}}2^{\frac{n}{\alpha}}\sum_{k=1}^{\infty}2^{k}e^{-b_{2}a_{4}kn}
=2n^{1+\frac{1}{\alpha}}2^{\frac{n}{\alpha}}\sum_{k=1}^{\infty}\left(\frac{2}{e^{b_{2}a_{4}n}}\right)^{k}\\
&\leq 2n^{1+\frac{1}{\alpha}}2^{\frac{n}{\alpha}}\frac{\frac{2}{e^{-b_{2}a_{4}n}}}{1-\frac{2}{e^{-b_{2}a_{4}n}}}
=2n^{1+\frac{1}{\alpha}}2^{\frac{n}{\alpha}}\frac{2}{e^{b_{2}a_{4}n}+2}\\
&\leq 6n^{1+\frac{1}{\alpha}}2^{\frac{n}{\alpha}}e^{-b_{2}a_{4}n}\leq 6n^{1+\frac{1}{\alpha}}e^{-n(b_{2}
a_{4}-\frac{\ln 2}{\alpha})}
\end{split}
\]
for all sufficiently large $n$'s. By taking $a_{4}$ so that $b_{2}a_{4}-\frac{\ln 2}{\alpha}>0$
we see from the Borel-Cantelli lemma that $\P\big(\limsup_{n}D_{n}\big)=0$.

By Lemma \ref{l:bsw} we have
$$
\P \Big(\sup_{0\leq t\leq 1}|X_{t}|=\infty \Big)=\lim_{n\to\infty}\P\Big(\sup_{0\leq t\leq 1}|X_{t}|>n\Big)
\leq \lim_{n\to\infty}c\psi^{*}(1/n)=0.
$$
Hence, $\P(\sup_{0\leq t\leq 1}|X_{t}|<\infty)=1$ and a.s. for each $\omega$ there exists $n$ 
such that $\sup_{0\leq t\leq 1}|X_{t}|\leq n$.
Hence, for $n$ sufficiently large, if $|y|>n$, then $L(y,[0,1])=0$.
For $|y|\leq n$, on the event $\big\{\limsup_{n}D_{n}\big\}^{c}$
we can express $y$ as $y=\lim_{k\to\infty}y_{k}$ with $y_{0}=x$ as before. 
Hence, for any $B\in \mathcal{D}_{n}$ and $x\in H_n$ we have
\begin{equation}\label{eqn:Holder continuity5}
\begin{split}
\left|L(x,B)-L(y,B)\right| &= \sum_{k=1}^{\infty}\left|L(y_{k}, B)- L(y_{k-1}, B)\right| \\
&\leq\sum_{k=1}^{\infty}(2^{-n})^{1-\frac{1+\gamma}{\alpha}}
|y_{k}-y_{k-1}|^{\gamma}(a_{4}kn)^{\frac{1+\gamma}{\alpha}} \\
&\leq \sum_{k=1}^{\infty}(2^{-n})^{1-\frac{1+\gamma}{\alpha}}
\left|2^{-\frac{n}{\alpha}} n^{-\frac{1}{\alpha}}2^{-k}\right|^{\gamma}(a_{4}kn)^{\frac{1+\gamma}{\alpha}} \\
&=a_{4}^{\frac{\gamma}{\alpha}}(2^{-n})^{1-\frac{1}
{\alpha}}n^{\frac{1}{\alpha}}
\sum_{k=1}^{\infty}2^{-k\gamma}k^{\frac{1+\gamma}{\alpha}} \leq c_{2}h(2^{-n})
\end{split}
\end{equation}
for some constant $c_{2}>0$. Now \eqref{eqn:Holder continuity2} follows from 
\eqref{eqn:Holder continuity4} and \eqref{eqn:Holder continuity5}.
\qed

The following lemma is an analogue of \cite[Lemma 3.1]{SXY}.
\begin{lemma}\label{lemma:covering2}
Suppose that $\psi\in\us(\overline{\alpha}, \overline{\theta}, \overline{C})$ and let 
$\gamma<\frac{1}{\overline{\alpha}}$.
Then, there exists a constant $K$ such that $\P^{x}$-a.s., for all  $n$
 sufficiently large, $X(I)$ can be covered by $K$ intervals of length 
 $2^{-n\gamma}$ for all $I\in \mathcal{C}_{n}$.
\end{lemma}
\pf
It follows from Lemma \ref{lemma:fluctuation} that for all $n$ such that 
$2^{-n}\leq \frac{1}{\overline{\theta}}$
$$
\P^{x}\bigg(\sup_{0\leq s\leq 2^{-n}}|X_{s}-x|\geq 2^{-n\gamma}\bigg)
\leq c2^{-n}(2^{-n\gamma})^{-\overline{\alpha}}
=c2^{-n(1-\gamma\overline{\alpha})}.
$$
Now the conclusion of the lemma follows from \cite[Lemma 2.1]{SXLJ}.
\qed

Now we are ready to prove Theorem \ref{Thm:H} $\textrm{ii)}$.

\noindent\textbf{Proof of Theorem \ref{Thm:H} ii)}

The proof is almost identical to the proof of \cite[Theorem 1.1]{SXY} 
using Theorem \ref{thm:HolderLC} and Lemma \ref{lemma:covering2}
instead of \cite[Lemmas 3.1 and 3.2]{SXY}
and we briefly sketch main steps. 
For any Borel set $F\subset \R$ one can find a probability measure $\mu$ supported on $F$ with $\mu(B)\leq \text{diam}(B)^{\dimh F-\eps}$ for any $B$ with $\text{diam}(B)\leq 1$. Define a measure $\lam$ supported on $\R^{+}$ as in \cite[Equation (3.2)]{SXY}. Then, using Theorem \ref{thm:HolderLC} and Lemma \ref{lemma:covering2} one can argue that the measure $\lam$ satisfies 
$\lam(B)\leq \text{diam}(B)^{1-\frac{1}{\alpha}+\gamma\dimh F-2\eps}$ for any $\gamma<1/\alpha$ and all Borel sets $B$ with sufficiently small diameter.  
This shows 
$$
\P^{x}\left(\dimh X^{-1}(F)\geq 1-\frac{1}{\alpha}+\gamma\dimh F -2\eps \text{ for all compact Borel sets } F\right)=1
$$
and letting $\gamma\uparrow \frac{1}{\alpha}$, then  $\eps\downarrow0$ establishes the claim.

\section{Examples}

In this section we provide some interesting examples to 
illustrate applications of the results of this paper. We recall 
from \cite{BGR} that a function $f:I\to \R$ is said to be almost 
increasing with a factor $c\in (0,1]$ if $cf(x)\leq  f(y)$ for 
all $x,y \in I$ and $x\leq y$. A function $f:I\to \R$ is said 
to be almost decreasing with a factor $C\in [1,\infty)$ if 
$Cf(x)\geq f(y)$ for all $x,y\in I$ and $x\leq y$. 
Finally, we recall the following characterizations for weakly scaling conditions.

\begin{lemma}{\cite[Lemma 11]{BGR}}\label{lemma:characterization}
We have $\phi\in \ls(\un{\alpha}, \un{\theta}, \un{c})$ if and only if 
$\phi(\theta)=\kappa(\theta)\theta^{\un{\alpha}}$ and
$\kappa$ is almost increasing on $(\un{\theta},\infty)$ with an 
oscillation factor $\un{c}$. Similarly, $\phi\in \us(\ov{\alpha},
\ov{\theta}, \ov{C})$ if and only if $\phi(\theta)=\kappa(\theta)
\theta^{\ov{\alpha}}$ and $\kappa$ is almost decreasing
on $(\ov{\theta},\infty)$ with an oscillation factor $\ov{C}$.
\end{lemma}
\begin{enumerate}
\item Symmetric stable processes

Let $X^{SS}$ be a symmetric stable L\'evy process in $\R$. 
The characteristic exponent of $X$ is $\psi^{SS}(\xi)
=|\xi|^{\alpha}$, $\alpha\in (0,2]$. When $\alpha=2$, $X^{SS}$  
is a Brownian motion whose sample paths are continuous,
which we exclude in this paper.

If $\alpha\in (1,2)$, then clearly $|\xi|^{\alpha}\in 
\us(\alpha,0,1)\cap \ls(\alpha,0,1)$. Hence, \eqref{eqn:main-equal} and \eqref{eqn:mainlbp} hold.

\item Relativistic stable processes

Let $X^{RS}$ be the relativistic stable process with mass 
$m$ in $\R$. The characteristic exponent of $X$ is given by
$$
\psi^{RS}(\xi)=(\xi^{2} +m^{2/\alpha})^{\alpha/2}-m, \quad 
\xi \in \R^{1}, m>0.
$$
Write
$$
\psi^{RS}(\xi)=|\xi|^{\alpha}\kappa_{1}(\xi), \quad 
\kappa_{1}(\xi)=\frac{(\xi^{2} +m^{2/\alpha})^{\alpha/2}-m}
{|\xi|^{\alpha}}.
$$
It is easy to check that
\beq\label{eqn:RS1}
\lim_{\xi\to\infty}\kappa_{1}(\xi)=1,
\eeq
and
\beq\label{eqn:RS2}
\lim_{\xi\to 0} \frac{\kappa_1(\xi)}{|\xi|^{2-\alpha}}=
\lim_{\xi\to 0} \frac{\left(m(\frac{|\xi|^{2}}
{m^{2/\alpha}}+1)^{2/\alpha}-m\right)}{|\xi|^{2}}
=\frac{2}{\alpha m^{2/\alpha}}.
\eeq
If follows from \eqref{eqn:RS1} and \eqref{eqn:RS2} 
$\kappa_1(\xi)$ is almost increasing on $(0,\infty)$ 
and is almost decreasing on $(\theta, \infty)$ for 
some $\theta>0$. This shows that $\psi^{RS}\in 
\ls(\alpha,0,\un{c})\cap \us(\alpha,\theta, \ov{C}_{1})$ 
from Lemma \ref{lemma:characterization}.
Hence, when $\alpha\in (1,2)$,  \eqref{eqn:main-equal} and \eqref{eqn:mainlbp} hold.

\item Truncated stable processes

Let $X^{TS}$ be the truncated stable L\'evy process. The characteristic 
exponent of $X^{TS}$ is given by
$$
\psi^{TS}(\xi)=\int_{\{0<|y|\leq 1\}}(1-\cos(y\xi))\frac{c(\alpha)}
{|y|^{1+\alpha}}dy,
$$
where $c(\alpha)$ is a constant so that $\int_{\R\setminus \{0\} }(1-\cos(y\xi))\frac{c(\alpha)}{|y|^{1+\alpha}}dy=1$.
By the change of variable $y=\frac{x}{|\xi|}$ we observe that
$$
\psi^{TS}(\xi)=c(\alpha)|\xi|^{\alpha}\int_{\{0<|x|\leq |\xi|\}}
\frac{1-\cos(\frac{\xi x}{|\xi|})}{|x|^{1+\alpha}}dx.
$$
Hence, we have $\psi^{TS}(\xi)\sim |\xi|^{\alpha}$ as $\xi\to \infty$. 
Since $1-\cos(\frac{\xi x}{|\xi|})\sim |x|^{2}$ as $|x|\to 0$, we 
observe that $\psi^{TS}(\xi)\sim c(\alpha)|\xi|^{2}$ as $\xi\to 0$.
Write $\psi^{TS}(\xi)$ as
$$
\psi^{TS}(\xi)=|\xi|^{\alpha}\kappa_{2}(\xi).
$$
Then, we observe that
$$
\lim_{\xi\to 0}\kappa_{2}(\xi)=\lim_{\xi\to 0}\frac{\psi^{TS}(\xi)}
{|\xi|^{\alpha}}=0 \text{ and } \lim_{\xi\to \infty}\kappa_{2}(\xi)=1.
$$
Hence, we see that $\kappa_{2}(\xi)$ is almost increasing on 
$(0,\infty)$ and is almost decreasing on $(\theta,\infty)$ for 
some $\theta>0$.
This shows that $\psi^{TS}(\xi)\in \ls(\alpha, 0, \un{c})
\cap \us(\alpha, \theta, \ov{C}_{1})$.
Hence we conclude that \eqref{eqn:main-equal} and \eqref{eqn:mainlbp} hold. 
\end{enumerate}

\bigskip
\noindent
{\bf Acknowledgements:} Research of Yimin Xiao was partially supported in 
part by the NSF grants DMS-1607089 and DMS-1855185. Xiaochuan Yang was 
supported in part by Luxembourg's National Foundation for Research (FNR), through the AFR project MiSSILe.

The authors thank the anonymous referee for providing some useful suggestions.

\begin{singlespace}

\end{singlespace}

\end{doublespace}

\vskip 0.3truein

{\bf Hyunchul Park}

Department of Mathematics, State University of New York at New Paltz, NY 12561,
USA

E-mail: \texttt{parkh@newpaltz.edu}

\bigskip

{\bf Yimin Xiao}

Dept. Statistics and Probability, Michigan State University, East Lansing, MI 48824, USA

E-mail: \texttt{xiaoyimi@stt.msu.edu}

\bigskip

{\bf Xiaochuan Yang}

Dept. Statistics and Probability, Michigan State University, East Lansing, MI 48824, USA

Mathematics Research Unit, University of Luxembourg, Grand Duchy of Luxembourg

E-mail: \texttt{xiaochuan.yang@uni.lu}

\end{document}